\documentclass{article}%
\usepackage{amssymb}
\usepackage{amsfonts}
\usepackage{tabulary}
\usepackage{amsmath}
\usepackage{textcomp}
\usepackage{multirow}
\usepackage{hyperref}
\usepackage{graphicx}%
\providecommand{\U}[1]{\protect\rule{.1in}{.1in}}

\newtheorem{theorem}{Theorem}
\newtheorem{definition}[theorem]{Definition}
\newtheorem{lemma}[theorem]{Lemma}
\newtheorem{example}[theorem]{Example}
\newtheorem{proposition}[theorem]{Proposition}
\newtheorem{corollary}[theorem]{Corollary}
\newtheorem{remark}[theorem]{Remark}
\newenvironment{proof}[1][Proof]{\noindent\textbf{#1. }}{\ $\square$\bigskip}
\begin{document}

\title{Induced Riemannian structures on null hypersurfaces}

\author{Manuel Guti\'{e}rrez$^{1}$\thanks{This paper was supported in part by MEYC-FEDER Grant MTM2013-41768-P and Junta de Andalucia research group FQM-324.}\\Departamento de \'Algebra, Geometr\'ia y Topolog\'ia\\Universidad de M\'alaga. Spain.\\ \ and Benjam\'{\i}n Olea$^{2}$ \\Departamento de Matem\'atica aplicada. Universidad de M\'alaga.\\ $^{1}$\texttt{mgl@agt.cie.uma.es}, $^{2}$\texttt{benji@agt.cie.uma.es} }

\date{}
\maketitle

\begin{abstract}
Given a null hypersurface $L$ of a Lorentzian manifold, we construct a
Riemannian metric $\widetilde{g}$ on it from a fixed transverse vector field
$\zeta$. We study the relationship between the ambient Lorentzian manifold,
the Riemannian manifold $(L,\widetilde{g})$ and the vector field $\zeta$. As
an application, we prove some new results on null hypersurfaces, as well as
known ones, using Riemannian techniques.
\end{abstract}

\textbf{keywords} Lorentzian manifolds, null hypersurface, rigging vector field, rigged vector field.\\

\textbf{msc2010} 53C50, 53C80, 53B30.

\section{Introduction}

\label{intro} Null hypersurfaces are exclusive objects from Lorentzian
manifolds, in the sense that they have not Riemannian counterpart, so they are
interesting by their own. There are also physical situations where they are
interesting objects since they represent light fronts in general relativity.
For example, in causality theory, a lower bound of the radius of injectivity
of null cones is important in understanding properties of solutions of wave
equations, \cite{ChenLeFloch2008,Klai-Rod2008}. They have been recently used
in the study of the formation of trapped surfaces and in the stability of
Minkowski space, \cite{Chris2009,ChristodoulouKlainerman1993} and they are
part of the quasi-local notions on black holes, which has been
introduced to understand black hole thermodynamic,
\cite{Ast99,Ast00,Ast02,GouJar06}. In a mathematical context, null
cones in Minkowski space are a key tool in the Fefferman-Graham construction to study
conformal invariants, which has been very influential in the celebrated
AdS/CFT correspondence, \cite{FeffermanGraham2008}. Moreover, it has been
suggested that null cones can be used in new
variants of the above mentioned AdS/CFT duality, \cite{Solodukhin2005}.

It is well-known that the main drawback to study null hypersurfaces as part of
standard submanifold theory is the degeneracy of the induced metric, which
forces to develop specific techniques. One of the most usual (but not the
unique) is to fix a geometric data formed by a null section and a screen
distribution (or equivalently a null section and a null transverse section) on the null hypersurface. This allows to define an induced
connection and a null second fundamental form, which gives the expected
information on the extrinsic geometry. However, the induced connection does
not arise necessarily from a metric and is clear that it is not an appropriate
tool to study intrinsic geometric properties. Moreover, both the null section
and the screen distribution are fixed arbitrarily and independently and it is
not clear how to choose them in order to have a reasonable coupling between
the properties of the null hypersuperface and the ambient space. Despite these limitations,
there are remarkable success, as those cited above.

In this paper we show a technique to construct a Riemannian metric on a null
hypersurface. It is based on the arbitrary choice of a transverse vector
field, called \textit{rigging field}, from which we construct a null section, which we call \textit{rigged field} and a screen
distibution. The improvement over the above technique is twofold: first, the geometric data depends only on the choice of
a unique object, the rigging field.  Secondly, we
introduce a Riemannian structure coupled with it, which is used to study the null hypersurface. Those structures are
not natural in the sense that they depend on the choice of the rigging field, but the flexibility
to choose it turns this limitation into an advantage, allowing us to use valuable information on the
ambient space, for example in the presence of symmetries. We cite Lemma \ref{campoU}, Corollary \ref{xi paralelo}, Theorem
\ref{Th. 2}, Theorem \ref{descomposicion}, Proposition \ref{puntosconjugados}  and Theorem \ref{Teor1} as examples.

The construction of the Riemannian metric $\widetilde{g}$ on a null
hypersurface $L$ is made in Section \ref{seccionproyeccion}. Roughly speaking,
the extrinsic properties of the null hypersurface are related to the
properties of the Riemannian manifold $(L,\widetilde{g})$. For example, if $H$
is the null mean curvature of $L$, then $H=-\widetilde{\operatorname{div}}\xi$
where $\xi$ is the rigged vector field. Moreover, $L$ is totally geodesic if
and only if $\xi$ is $\widetilde{g}$-orthogonally Killing. The key point is that we can tune the geometry of the ambient
Lorentzian manifold $(M,g)$ and $(L,\widetilde{g})$ for each situation. For example, it is classical to renormalize the null vector field $\xi$ so that it becomes geodesic.
In our approach we can achieve  a geodesic rigged field  if the rigging is conformal, Lemma \ref{campoU}, or if $L$ is a null cone and $\xi$ is the gradient of the time coordinate of a normal chart, Theorem \ref{Teor4}. Moreover, we can use the conformal symmetry of the ambient space to prove new results on the geometry of null hypersurfaces using Riemannian techniques, Theorem \ref{Th. 2}.

In Section \ref{seccioncurvatura}, we establish some formulas linking the
curvature of the ambient manifold and the curvature of $(L,\widetilde{g})$.
These kind of relations are necessary. They allow us to obtain new properties
of null hypersurfaces and they are used in Section
\ref{seccionaplicaciones}, where we illustrate our ideas with some
applications. In Theorem \ref{Th. 2} mentioned above we use the Bochner technique to show a
curvature condition which implies that a compact totally umbilic null hypersurface
must be totally geodesic. Theorem \ref{descomposicion} shows that the induced
Riemannian metric $\widetilde{g}$ in a totally umbilic null hypersurface is
locally a twisted product, which can be a warped or direct product depending on
the properties of the ambient space and the rigging field, giving a new
insight to twisted metrics in Lorentzian geometry. We show an example of how we can tune both geometries to study null conjugate points.
This is done in Proposition \ref{puntosconjugados} where a close link between conjugate points along null geodesics as seen in the ambient space and
in a null hypersurface with our Riemannian metric is given. This allows us to give a new proof, using Riemannian techniques, of a localization result
for conjugate points along null geodesic, Theorem \ref{Teor4}. These ideas are used to prove that the first conjugate point of a
null geodesic contained in a totally umbilic
null cone has maximum multiplicity, Theorem \ref{Teor1}. Finally, in Section
\ref{NullSurfaces}, we adapt the main ideas
to null submanifods, which allows us to apply Gauss-Bonnet theorem to compact
null surfaces.

Summing up, the rigging technique introduced in this paper is a convenient way to handle the technique of introducing a geometric
data on $L$, so we can get at least the same results obtained with it. Moreover we have some extra benefits, like a Riemannian structure on $L$ with an adequate tuning with the ambient geometry as well as with the geometric data, which represents an objective advance in the theory. In fact, it allows us to think with classical Riemannian tools. This lead us to reprove well known results in the literature, e. g. Theorem \ref{Teor4} and to get new results, e. g. Proposition \ref{Prop1}, Theorems \ref{Th. 2} and Theorem \ref{Teor1}.

\section{Geometry of null hypersurfaces}

\label{review} We review some facts about null hypersurfaces to fix notations
(see \cite{DuggaBejan1996} for details).

Given $(M,g)$ a $n$-dimensional time-orientable Lorentzian manifold and $L$ a
null hypersurface, we can choose a null vector field $\xi\in\mathfrak{X}(L)$.
A \emph{screen distribution} $\mathcal{S}$ is a complementary distribution in
$TL$ to $\xi$ and the \emph{transverse distribution} is the unique
one-dimensional null distribution orthogonal to $\mathcal{S}$ not contained in
$TL$. Being $M$ time-orientable, there is a null vector field $N$ over $L$
which generates the transverse distribution and it can be normalized so that
$g(N,\xi)=1$. It is called \emph{null transverse vector field}.

We usually denote by $U,V,W$ vector fields in $L$ and $X,Y,Z$ vector fields in
$\mathcal{S}$. If $U,V\in\mathfrak{X}(L)$, the vector fields $\nabla_{U}V$ and
$\nabla_{U}N$ can be decomposed in the following way.
\begin{align}
\nabla_{U}V  &  =\nabla_{U}^{L}V+B(U,V)N,\label{descomp1}\\
\nabla_{U}N  &  =\tau(U)N-A(U),\nonumber
\end{align}
where $\nabla_{U}^{L}V,A(U)\in TL$ and $\tau$ is a one-form. The operator
$\nabla^{L}$ is a symmetric connection that satisfies
\[
\left(  \nabla_{U}^{L}g\right)  (V,W)=B(U,V)g(N,W)+B(U,W)g(N,V),
\]
$A$ is the \emph{shape operator} of $L$ and $B$ is a symmetric tensor, called
the \emph{second fundamental form} of $L$, that satisfies $B(U,V)=-g(\nabla
_{U}\xi,V)$. Moreover, $B(\xi,\cdot)=0$ and $\xi$ is a pregeodesic vector
field, in fact $\nabla_{\xi}\xi=-\tau(\xi)\xi$.

The notion of totally geodesic or umbilic hypersurface also has sense in the
degenerate case and they do not depend on the election of the null section neither the screen distribution. Indeed, $L$ is \emph{totally geodesic} if $B\equiv0$ and
totally umbilic if $B=\rho g$ for certain $\rho\in C^{\infty}(L)$.

Given $U\in\mathfrak{X}(L)$, the vector field $\nabla_{U}\xi$ belongs to $L$,
so it can be decomposed as
\[
\nabla_{U}\xi=-\tau(U)\xi-A^{\ast}(U),
\]
where $A^{\ast}(U)\in\mathcal{S}$. The endomorphism $A^{\ast}$ is called the
\emph{shape operator} of $\mathcal{S}$ and it satisfies $B(U,V)=g(A^{\ast
}(U),V)$ and
\begin{equation}
B(A^{\ast}(U),V)=B(U,A^{\ast}(V)). \label{ABadjunto}%
\end{equation}

The trace of $A^{\ast}$ is the \emph{null mean curvature} of $L$, explicitly
given by
\[
H_{p}=\sum_{i=3}^{n}g(A^{\ast}(e_{i}),e_{i})=\sum_{i=3}^{n}B(e_{i},e_{i}),
\]
being $\{e_{3},\ldots,e_{n}\}$ an orthonormal basis of $\mathcal{S}_{p}$.

On the other hand, given $U\in\mathfrak{X}(L)$ and $X\in\mathcal{S}$, we
decompose
\begin{equation}
\nabla_{U}^{L}X=\nabla_{U}^{\ast}X+C(U,X)\xi, \label{descomp2}%
\end{equation}
where $\nabla_{U}^{\ast}X\in \mathcal{S}$.
The tensor $C$ holds $C(U,X)=-g(\nabla_{U}N,X)=g(A(U),X)$ and
\[
C(X,Y)-C(Y,X)=g(N,[X,Y]).
\]

In case of being $\mathcal{S}$ integrable, $\nabla^{\ast}$ is the induced
Levi-Civita connection from $(M,g)$ and Equations (\ref{descomp1}) and
(\ref{descomp2}) show that its second fundamental form is
\begin{equation}
\mathbb{I}^{\mathcal{S}}(X,Y)=C(X,Y)\xi+B(X,Y)N,\label{Eq3}
\end{equation}
where $X,Y\in\mathcal{S}$.

The \emph{curvature tensor of} $\nabla^{L}$ is defined as $R_{UV}^{L}%
W=\nabla_{U}^{L}\nabla_{V}^{L}W-\nabla_{V}^{L}\nabla_{U}^{L}W-\nabla_{\lbrack
U,V]}^{L}W$ and it satisfies
\begin{equation}
R_{UV}^{L}\xi=R_{UV}\xi\label{eqcurv}%
\end{equation}
and the so called \emph{Gauss-Codazzi equations}
\begin{align}
g(R_{UV}W,X)  &  =g(R_{UV}^{L}W,X)+B(U,W)g(A(V),X)\label{eqcurv2}\\
&  -B(V,W)g(A(U),X),\nonumber\\
g(R_{UV}W,\xi)  &  =\left(  \nabla_{U}^{L}B\right)  (V,W)-\left(  \nabla
_{V}^{L}B\right)  (U,W)+\tau(U)B(V,W)\label{eqcurvatura1}\\
&  -\tau(V)B(U,W),\nonumber\\
g(R_{UV}W,N)  &  =g(R_{UV}^{L}W,N),\nonumber
\end{align}
where $U,V,W\in\mathfrak{X}(L)$ and $X\in\mathcal{S}$. From these equations it
can be deduced the following ones.
\begin{align}
g(R_{UV}X,N)  &  =\left(  \nabla_{U}^{\ast L}C\right)  (V,X)-\left(
\nabla_{V}^{\ast L}C\right)  (U,X)+\tau(V)C(U,X)\nonumber\\
&  -\tau(U)C(V,X),\label{GC2}\\
g(R_{UV}\xi,N)  &  =C(V,A^{\ast}(U))-C(U,A^{\ast}(V))-d\tau(U,V), \label{GC3}%
\end{align}
where $\nabla_{U}^{\ast L}C$ is defined as
\[
\left(  \nabla_{U}^{\ast L}C\right)  (V,X)=U(C(V,X))-C(\nabla_{U}%
^{L}V,X)-C(V,\nabla_{U}^{\ast}X).
\]

Using Equation (\ref{eqcurvatura1}), we can compute the null sectional
curvature respect to $\xi$ of a null plane $\Pi=span(X,\xi)$, where
$X\in\mathcal{S}$ is unitary,

\begin{equation}
\mathcal{K}_{\xi}(\Pi)=\left(  \nabla_{\xi}^{L}B\right)  (X,X)-\left(
\nabla_{X}^{L}B\right)  (\xi,X)+\tau(\xi)B(X,X). \label{Eq1}
\end{equation}

In particular, if $L$ is totally geodesic, we have $\mathcal{K}_{\xi}(\Pi)=0$
for any null tangent plane $\Pi$ to $L$. Since the sign of the null sectional
curvature only depends on the null plane and not on the choosen null vector,
we can state the following.

\begin{proposition}
Let $M$ be a Lorentzian manifold and $p\in M$ such that
$\mathcal{K}(\Pi)\neq0$ for any null plane $\Pi\subset T_{p}M$. Then, it does
not exist any totally geodesic null hypersurface through $p$.
\end{proposition}

As a simple application of this result, the Friedmann models do not possess
totally geodesic null hypersurfaces, since its null sectional curvature never
vanishes (Corollary 6.5 of \cite{GutOle09}).

\section{Riemannian metric induced on a null hypersurface}

\label{seccionproyeccion}

In this section, we show how to induce a Riemannian metric on a null
hypersurface $L$. The construction depends on the choice of a vector field in
a neighborhood of $L$ and, as we will see, it also induces a null vector field
and a screen distribution on $L$.

Take $\zeta$ a vector field defined in some open set containig $L$ and denote
by $\alpha$ the $1$-form metrically equivalent to $\zeta$. Take $\omega
=i^{\ast}\alpha$, being $i:L\hookrightarrow M$ the canonical inclusion, and
consider the tensors $\overline{g}=g+\alpha\otimes\alpha$ and $\widetilde
{g}=i^{\ast}\overline{g}$.

\begin{lemma}
\label{lema1} Given a point $p\in L$, the following statements hold.

\begin{enumerate}
\item $\overline{g}_{p}$ is degenerate if and only if $\zeta_{p}$ is timelike
and unitary for $g$.

\item $\widetilde{g}_{p}$ is Riemannian if and only if $\zeta_{p}\notin
T_{p}L$.
\end{enumerate}
\end{lemma}

If $\zeta_{p}$ is timelike and $\left\vert \zeta_{p}\right\vert >1$ (resp.
$<1)$, then $\overline{g}_{p}$ is Riemannian (resp. Lorentzian). The
Riemannian metric $\widetilde{g}$ on $L$ can be written as $\widetilde
{g}=i^{\ast}g+\omega\otimes\omega$, and it is clear that we can construct a
Lorentzian metric on $L$ defining $\widetilde{g}=i^{\ast}g-\omega\otimes
\omega$, but we will only consider the Riemannian case in this paper.

Point 2 of the above lemma motivates us to define the following.

\begin{definition}
Let $L$ be a null hypersurface of a Lorentzian manifold. A \emph{rigging} for $L$ is
a vector field $\zeta$ defined on some open set containing $L$ such that
$\zeta_{p}\notin T_{p}L$ for each $p\in L$.
\end{definition}

As far as we know, the term rigging appeared for the first time in \cite{Schou1954} in a Riemannian setting. In \cite{1980Katsuno,1981Katsuno},
the author used the term rigging to refer to a null vector field transverse to a null hypersurface and used it to construct a Riemannian metric, see also \cite{Atindogbe}.

From now on we fix $\zeta$ a rigging for $L$ and we induce a Riemannian metric
$\widetilde{g}$ on $L$ which we call \emph{rigged metric}, as Lemma \ref{lema1} asserts. It also fixes a null
vector field in $L$, as we can see below.

\begin{definition}
The \emph{rigged vector field} of $\zeta$ is the $\widetilde{g}$-metrically
equivalent vector field to the $1$-form $\omega$ and it is denoted by $\xi$.
\end{definition}

\begin{lemma}
\label{lemacaractproyeccion} The rigged vector field $\xi$ is the unique null
vector field in $L$ such that $g(\zeta,\xi)=1$. Moreover, $\xi$ is
$\widetilde{g}$-unitary.
\end{lemma}

\begin{proof}
Take $v\in TL$ a null vector. Since $\zeta_{p}\notin TL$ for each $p\in L$,
using the definitions of $\omega$ and $\widetilde{g}$, we have $\widetilde
{g}(\xi,v)=g(\zeta,v)\neq0$ and $\widetilde{g}(\xi,v)=\widetilde{g}(\xi
,\xi)\widetilde{g}(\xi,v)$, thus $\widetilde{g}(\xi,\xi)=1$. Now, observe that
$\widetilde{g}(\xi,\xi)=g(\zeta,\xi)$ and $\widetilde{g}(\xi,\xi)=g(\xi
,\xi)+\widetilde{g}(\xi,\xi)^{2}$.
\end{proof}

We can consider the screen distribution given by $TL\cap\zeta^{\perp}$, which
we denote by $\mathcal{S}^{\zeta}$ to emphasize that it depends on $\zeta$, and derive all geometrical objects defined
in Section \ref{review}. Observe that $\mathcal{S}^{\zeta}$ is the
$\widetilde{g}$-orthogonal subspace to $\xi$ and the null transverse vector
field to $\mathcal{S}^{\zeta}$ is given by
\[
N=\zeta-\frac{1}{2}g(\zeta,\zeta)\xi.
\]

Using Lemma \ref{lema1}, we could have constructed a Riemannian metric on the
whole $M$ which would induce a Riemannian metric on $L$. However, this
construction is too rigid and it restricts the link between the geometry of
the ambient and the geometry induced on $L$. In this paper, the key point is
not the Riemannian character of the metric $\overline{g}$ on the ambient, but
that of the induced metric $\widetilde{g}$ on $L$. In a similar way we could have defined the rigging vector field $\zeta$ on $L$ instead on an open set containing $L$, but the tuning of both geometries are easier described with hypothesis on the properties of $\zeta$ as a vector field on $M$.

We need the relation between the Levi-Civita connections of both $g$ and
$\widetilde{g}$ acting on vector fields in $\mathfrak{X}(L)$. Call
$\widetilde{\nabla}$ the Levi-Civita connection induced on $L$ by
$\widetilde{g}$ and $D=\nabla-\widetilde{\nabla}$, which is a symmetric tensor
on $\mathfrak{X}(L)$.

\begin{proposition}
\label{diferenciaconexiones} Given $U,V,W\in\mathfrak{X}(L)$, it holds
\[
g(D(U,V),W)=-\frac{1}{2}\left(  \omega(W)(L_{\xi}\widetilde{g})(U,V)+\omega
(U)d\omega(V,W)+\omega(V)d\omega(U,W)\right)  ,
\]
where $L_{\xi}$ is the Lie derivative along $\xi$.
\end{proposition}

\begin{proof}
We can suppose that the involved Lie brackets vanish. The Koszul identity
leads us to write
\begin{align*}
2g(\nabla_{U}V,W)  &  =2\widetilde{g}(\widetilde{\nabla}_{U}V,W)-\Big(U\left(
\omega\otimes\omega(V,W)\right)  +V\left(  \omega\otimes\omega(U,W)\right) \\
&  -W\left(  \omega\otimes\omega(U,V)\right)  \Big)\\
&  =2g(\widetilde{\nabla}_{U}V,W)\\
&  -\left(  \widetilde{\nabla}_{U}(\omega\otimes\omega)(V,W)+\widetilde
{\nabla}_{V}(\omega\otimes\omega)(U,W)-\widetilde{\nabla}_{W}(\omega
\otimes\omega)(U,V)\right)  .
\end{align*}

Now, just take into account that $d\omega(U,V) =\left(  \widetilde{\nabla}%
_{U}\omega\right)  (V)-\left(  \widetilde{\nabla}_{V}\omega\right)  (U)$ and
$(L_{\xi}\widetilde{g})(U,V) =\left(  \widetilde{\nabla}_{U}\omega\right)
(V)+\left(  \widetilde{\nabla}_{V}\omega\right)  (U)$.
\end{proof}

Now, we take $D^{L}=\nabla^{L}-\widetilde{\nabla}$, which is also symmetric
and holds $D-D^{L}=B\cdot N$. Therefore
\begin{align}
g(D^{L}(U,V),W)  &  =-\frac{1}{2}\left(  \omega(W)(L_{\xi}\widetilde
{g})(U,V)+\omega(U)d\omega(V,W)+\omega(V)d\omega(U,W)\right) \nonumber\\
&  -B(U,V)\omega(W). \label{tensordiferencia}%
\end{align}

The fact that both $\nabla^{L}$ and $\widetilde{\nabla}$ are connections on
$L$ makes the computations in the next section easier with $D^{L}$ instead of
$D$. The following basic identities holds.

\begin{corollary}
\label{corformula}Given $U\in\mathfrak{X}(L)$ and $X,Y,Z\in\mathcal{S}^{\zeta
}$, we have the following.

\begin{enumerate}
\item $\widetilde{g}(D^{L}(X,U),X)=g(D^{L}(X,U),X)=0$.

\item $\widetilde{g}(D^{L}(X,Y),Z)=g(D^{L}(X,Y),Z)=0$.

\item $\widetilde{g}(D^{L}(U,\xi),\xi)=-\tau(U)=-g(\nabla_{U}\zeta,\xi)$.

\item $-2C(U,X)=d\alpha(U,X)+\left(  L_{\zeta}g\right)  (U,X)+g(\zeta
,\zeta)B(U,X)$.
\end{enumerate}
\end{corollary}

Now, we relate the null mean curvature of the null hypersurface $L$ with
$\widetilde{g}$.

\begin{proposition}
\label{derivadadeLie} Take $X,Y,Z\in\mathcal{S}^{\zeta}$. It holds

\begin{enumerate}
\item $\widetilde{\nabla}_{X}Y=\nabla_{X}^{\ast}Y-\widetilde{g}(\widetilde
{\nabla}_{X}\xi,Y)\xi$, thus $\widetilde{g}(\widetilde{\nabla}_{X}%
Y,Z)=g(\nabla_{X}Y,Z)$.

\item $\left(  L_{\xi}\widetilde{g}\right)  (X,Y)=-2B(X,Y)$. In particular
$H=-\widetilde{\operatorname{div}}\xi$.
\end{enumerate}
\end{proposition}

\begin{proof}
From the above corollary $D^{L}(X,Y)=a\xi$ for certain $a$. Thus
$\widetilde{\nabla}_{X}Y=\nabla_{X}^{\ast}Y+\left(  C(X,Y)-a\right)  \xi$, but
$C(X,Y)-a=\widetilde{g}(\widetilde{\nabla}_{X}Y,\xi)=-\widetilde{g}%
(\widetilde{\nabla}_{X}\xi,Y)$. Second point follows taking $U=X,V=Y$ and
$W=\xi$ in Formula (\ref{tensordiferencia}).
\end{proof}

\begin{remark}
\label{remark2}From the first point of the above proposition, for all
$X,Y\in\mathcal{S}^{\zeta}$,
\begin{equation}
D^{L}(X,Y)=\left(  C(X,Y)+\widetilde{g}(\widetilde{\nabla}_{X}\xi,Y)\right)
\xi. \label{remarkD}%
\end{equation}

On the other hand, from the second point, $L$ is totally geodesic if and only
if $\xi$ is $\widetilde{g}$-orthogonally Killing and it is totally umbilic if
and only if $\xi$ is $\widetilde{g}$-orthogonally conformal.
\end{remark}

\begin{remark}
Suppose that $M$ is orientable and $L$ is a compact null hypersurface. Then
$L$ is also orientable, and point 2 above implies $\int_{L}Hd\widetilde{g}=0$.
In particular, its null mean curvature vanishes somewhere. This fact is a
remarkable difference with respect to non-null hypersurfaces.
\end{remark}

A classical question is to find conditions on a null hypersurface to ensure
that the choosen null section on it can be rescaled to be geodesic, see for
example \cite{Kup87-2}. It can always be done locally, but in general it is not possible to do it globally as it can be checked in
$\mathbb{T}\times\mathbb{R}$, where $\mathbb{T}$ is the Clifton-Pohl torus.
The following lemma shows that with our approach this
question is naturally answered giving a condition in terms of the rigging
vector field.

\begin{lemma}
\label{campoU} Let $L$ be a null hypersurface and $\zeta$ a rigging for it. If
$\zeta$ is conformal, then $\nabla_{\xi}\xi=0$, that is $\tau(\xi)=0$. Moreover
$\tau(X)=-\frac{1}{2}g(\widetilde{\nabla}_{\xi}\xi,X)$ for all $X\in
\mathcal{S}^{\zeta}$.
\end{lemma}

\begin{proof}
From Corollary \ref{corformula}, $\tau(\xi)=g(\nabla_{\xi}\zeta,\xi)=0$. On
the other hand, Proposition \ref{diferenciaconexiones} and Corolary
\ref{corformula} implies
\begin{align*}
g(\widetilde{\nabla}_{\xi}\xi,X)  &  =-g(D(\xi,\xi),X)=d\omega(\xi
,X)=d\alpha(\xi,X)=-2g(\xi,\nabla_{X}\zeta)\\
&  =-2\tau(X).
\end{align*}

\end{proof}

The above lemma allows us to state the following.

\begin{proposition} \label{Prop1}
Let $M$ be a null complete Lorentzian manifold furnished with a timelike
conformal vector field. If $Ric(u)>0$ for all null vector $u\in TM$,
then it cannot exist any  closed (in the topological sense) embedded null hypersurface.
\end{proposition}
\begin{proof}
Suppose that $L$ is a closed embedded null hypersurface. Since the rigged vector field
$\xi$ is geodesic (Lemma \ref{campoU}) and $M$ is null complete, it follows
that $\xi$ is a complete vector field. From the well-known Raychaudhuri equation (see for example
\cite[Proposition 2]{2005Duggal})
\begin{equation*}
 Ric(\xi)=\xi(H)-|A^*|^2
\end{equation*}
  and the inequality $\frac{1}{n-2}\left(tr A^*\right)^2\leq|A^*|^2$, where $|A^*|$ is the norm of $A^*$, we have
  $0< \xi(H)-\frac{H^2}{n-2}$. Since  $\xi$ is complete,
it follows that $H=0$ which is a contradiction.
\end{proof}

\begin{example}
\label{ejemploBerger} Using this proposition, the Lorentzian Berger sphere
$(\mathbb{S}^{2n+1},g_{L})$ does not admit any closed embedded null hypersurface. In
fact, it is a complete Lorentzian manifold, $Ric(u)>0$ for all null
vector $u\in TM$  and the Hopf vector field is timelike and Killing, \cite{Gut2}.
\end{example}

If we take a closed rigging $\zeta$, its rigged vector field $\xi$ is also
closed, so the screen distribution $\mathcal{S}^{\zeta}$ is integrable. We
call $\widetilde{\mathbb{I}}$ the second fundamental form of $(\mathcal{S}%
^{\zeta},\widetilde{g})$ inside $(L,\widetilde{g})$.

\begin{proposition}
\label{segundaformafund}Let $L$ be a null hypersurface and $\zeta$ a closed
rigging for it with rigged vector field $\xi$. Given $X,Y\in\mathcal{S}%
^{\zeta}$ and $U\in\mathfrak{X}(L)$ it holds
\begin{align*}
\widetilde{\nabla}_{X}Y  &  =\nabla_{X}^{\ast}Y+B(X,Y)\xi,\\
\widetilde{\nabla}_{U}\xi &  =-A^{\ast}(U).
\end{align*}

In particular, $\widetilde{\nabla}_{\xi}\xi=0$ and $\widetilde{\mathbb{I}%
}(X,Y)=B(X,Y)\xi$.
\end{proposition}

\begin{proof}
Being $\xi$ closed and unitary, $\widetilde{\nabla}_{\xi}\xi=0$ and $\left(
L_{\xi}\widetilde{g}\right)  (X,Y)=2\widetilde{g}(\widetilde{\nabla}_{X}%
\xi,Y)$. From Proposition \ref{derivadadeLie}, we have $B(X,Y)=-\widetilde
{g}(\widetilde{\nabla}_{X}\xi,Y)$ and $\widetilde{\nabla}_{X}Y=\nabla
_{X}^{\ast}Y+B(X,Y)\xi$. Moreover, since $B(X,Y)=g(A^{\ast}(X),Y)$, it follows
$\widetilde{\nabla}_{X}\xi=-A^{\ast}(X)$.
\end{proof}

An inmediate consequence of the above proposition is the following.

\begin{corollary}
\label{xi paralelo}Let $L$ be a null hypersurface and $\zeta$ a closed rigging
for it.

\begin{enumerate}
\item $L$ is totally geodesic if and only if the rigged vector field $\xi$ is
$\widetilde{g}$-parallel.

\item $L$ is totally geodesic (resp. umbilic) if and only if each leaf of
$\mathcal{S}^{\zeta}$ is totally geodesic (resp. umbilic) as a hypersurface of
$(L,\widetilde{g})$.
\end{enumerate}
\end{corollary}

Compare point 1 and 2 above with Remark \ref{remark2}. On the other hand, in general, if the leaves of an integrable screen distribution are
totally umbilic in $(M,g)$, then $L$ is totally umbilic, see Equation (\ref{Eq3}). The converse does not hold and this lack of symmetry hide the geometric
meaning of umbilicity in the null case. However, the converse does hold in $(L,\widetilde{g})$, which suggests the convenience of the rigging construction.

\begin{proposition}
\label{Nablas} Let $L$ be a null hypersurface and $\zeta$ a closed rigging for
it. Then $\widetilde{g}(\widetilde{\nabla}_{U}V,W)=g(\nabla_{U}V,W)+\omega
(W)U(\omega(V))$ for all $U,V,W\in\mathfrak{X}(L)$.
\end{proposition}

\begin{proof}
Given $X,Y,Z\in\mathcal{S}^{\zeta}$, from Proposition \ref{derivadadeLie} and
\ref{segundaformafund}, $\widetilde{g}(\widetilde{\nabla}_{X}Y,Z)=g(\nabla
_{X}Y,Z)$ and $\widetilde{g}(\widetilde{\nabla}_{X}\xi,Y)=g(\nabla_{X}\xi,Y)$
respectively. Using these equations it can be checked that $\widetilde
{g}(\widetilde{\nabla}_{U}Y,W)=g(\nabla_{U}Y,W)$ and $\widetilde{g}%
(\widetilde{\nabla}_{U}\xi,W)=g(\nabla_{U}\xi,W)$ for all $U,W\in
\mathfrak{X}(L)$ and $Y\in\mathcal{S}^{\zeta}$. Now, if we take $V=\omega
(V)\xi+Y$, then
\begin{align*}
\widetilde{g}(\widetilde{\nabla}_{U} V,W)  &  =U(\omega(V))\omega(W)
+\omega(V)\widetilde{g}(\widetilde{\nabla}_{U}\xi,W)+\widetilde{g}%
(\widetilde{\nabla}_{U}Y,W)\\
&  =U(\omega(V))\omega(W)+g(\nabla_{U} V,W).
\end{align*}

\end{proof}

\section{Curvature relations}

\label{seccioncurvatura} In this section we relate the curvature tensor
$R^{L}$ derived from the linear connection $\nabla^{L}$ and the curvature
tensor $\widetilde{R}$ of $(L,\widetilde{g})$ as a Riemannian manifold itself.
Using Gauss-Codazzi equations, we can also relate the curvature of $(M,g)$ and
$(L,\widetilde{g})$. We use the following well-known general result.

\begin{lemma}
\label{diferenciacurvaturas} Let $R^{L}$, $\widetilde{R}$ be the curvature
tensors associated to arbitrary symmetric connections $\nabla^{L}$,
$\widetilde{\nabla}$ on a manifold $L$. Given $U,V,W\in\mathfrak{X}(L)$ it
holds
\begin{align*}
R_{UV}^{L}W  &  =\widetilde{R}_{UV}W+(\widetilde{\nabla}_{U}D^{L}%
)(V,W)-(\widetilde{\nabla}_{V}D^{L})(U,W)\\
&  +D^{L}(U,D^{L}(V,W))-D^{L}(V,D^{L}(U,W)),
\end{align*}
where $D^{L}=\nabla^{L}-\widetilde{\nabla}$.
\end{lemma}

First, we relate the sectional curvatures of $\widetilde{g}$-orthogonal planes
to $\xi$. In this section we need the adjoint of $\widetilde{\nabla}_{U}\xi$
as an endomorphism, so we will use the notation $S(U)=\widetilde{\nabla}%
_{U}\xi$ for simplicity.

\begin{theorem}
\label{diferenciacurvatura3} Let $M$ be a Lorentzian manifold, $L$ a null
hypersurface and $\zeta$ a rigging for it. If $\Pi=span(X,Y)$, being
$X,Y\in\mathcal{S}^{\zeta}$ unitary and orthogonal vectors, then
\begin{align*}
K(\Pi)-\widetilde{K}(\Pi)  &  =-C(Y,Y)B(X,X)-C(X,X)B(Y,Y)\\
&  +\left(  C(X,Y)+C(Y,X)\right)  B(X,Y)\\
&  +B(X,X)B(Y,Y)-B(X,Y)^{2}+\frac{3}{4}d\omega(X,Y)^{2}.
\end{align*}

\end{theorem}

\begin{proof}
From Lemma \ref{diferenciacurvaturas},
\begin{align*}
\widetilde{g}(R_{XY}^{L}Y-\widetilde{R}_{XY}Y,X)  &  =\widetilde
{g}((\widetilde{\nabla}_{X}D^{L})(Y,Y),X)-\widetilde{g}((\widetilde{\nabla
}_{Y}D^{L})(X,Y),X)\\
&  +\widetilde{g}(D^{L}(X,D^{L}(Y,Y)),X)-\widetilde{g}(D^{L}(Y,D^{L}(X,Y)),X).
\end{align*}

We compute each term. Using Formulas (\ref{tensordiferencia}) and
(\ref{remarkD}), the first one is
\[
\widetilde{g}((\widetilde{\nabla}_{X}D^{L})(Y,Y),X)=\left(  C(Y,Y)+\widetilde
{g}(S(Y),Y)\right)  \widetilde{g}(S(X),X)+g(S(X),Y)d\omega(X,Y).
\]

The second term is computed in a similar way.
\[
\widetilde{g}((\widetilde{\nabla}_{Y}D^{L})(X,Y),X)=\left(  C(X,Y)+\widetilde
{g}(S(X),Y)\right)  \widetilde{g}(S(Y),X)+\frac{1}{2}\widetilde{g}%
(S(Y),X)d\omega(X,Y).
\]

The third one vanishes by Corollary \ref{corformula}. We compute the last
one.
\begin{align*}
\widetilde{g}(D^{L}(Y,D^{L}(X,Y)),X)  &  =-\frac{1}{2}\widetilde{g}%
(D^{L}(X,Y),\xi)d\omega(Y,X)\\
&  =\frac{1}{2}\left(  C(X,Y)+\widetilde{g}(S(X),Y)\right)  d\omega(X,Y).
\end{align*}

Using the identities $d\omega(X,Y)=\widetilde{g}(S(X),Y)-\widetilde
{g}(X,S(Y))$ and $\left(  L_{\xi}\widetilde{g}\right)  (X,Y)=\widetilde
{g}(S(X),Y)+\widetilde{g}(X,S(Y))$ we have
\begin{align*}
\widetilde{g}(R_{XY}^{L}Y-\widetilde{R}_{XY}Y,X)  &  =\frac{1}{2}C(Y,Y)\left(
L_{\xi}\widetilde{g}\right)  (X,X)+\frac{1}{4}\left(  L_{\xi}\widetilde
{g}\right)  (Y,Y)\left(  L_{\xi}\widetilde{g}\right)  (X,X)\\
&  -\frac{1}{2}C(X,Y)\left(  L_{\xi}\widetilde{g}\right)  (X,Y)-\widetilde
{g}(S(X),Y)\widetilde{g}(S(Y),X)+\frac{1}{2}d\omega(X,Y)^{2}.
\end{align*}

We can express
\begin{align*}
\widetilde{g}(S(X),Y)\widetilde{g}(S(Y),X)  &  =\frac{1}{4}\left(  \left(
L_{\xi}\widetilde{g}\right)  (X,Y)+d\omega(X,Y)\right)  \left(  \left(
L_{\xi}\widetilde{g}\right)  (X,Y)-d\omega(X,Y)\right) \\
&  =\frac{1}{4}\left(  L_{\xi}\widetilde{g}\right)  (X,Y)^{2}-\frac{1}%
{4}d\omega(X,Y)^{2},
\end{align*}
thus
\begin{align*}
\widetilde{g}(R_{XY}^{L}Y-\widetilde{R}_{XY}Y,X)  &  =\frac{1}{2}C(Y,Y)\left(
L_{\xi}\widetilde{g}\right)  (X,X)+\frac{1}{4}\left(  L_{\xi}\widetilde
{g}\right)  (X,X)\left(  L_{\xi}\widetilde{g}\right)  (Y,Y)\\
&  -\frac{1}{2}C(X,Y)\left(  L_{\xi}\widetilde{g}\right)  (X,Y)-\frac{1}%
{4}\left(  L_{\xi}\widetilde{g}\right)  (X,Y)^{2}+\frac{3}{4}d\omega(X,Y)^{2}.
\end{align*}

Finally, using Proposition \ref{derivadadeLie} and the Gauss-Codazzi equation
(\ref{eqcurv2}), we get the result.
\end{proof}

Observe that if $L$ is totally geodesic, then $K(\Pi)\geq\widetilde{K}(\Pi)$
for any tangent plane contained in $\mathcal{S}^{\zeta}$.

Now, consider $S^{\ast}:\mathfrak{X}(L)\rightarrow\mathfrak{X}(L)$ the adjoint
endomorphism of $S$. We can decompose $S^{\ast}(U)$ as
\begin{equation}
S^{\ast}(U)=S^{\ast\perp}(U)+\widetilde{g}(\widetilde{\nabla}_{\xi}\xi,U)\xi,
\label{operadorT}%
\end{equation}
where $S^{\ast\perp}(U)$ is $\widetilde{g}$-orthogonal to $\xi$. Observe that
$S^{\ast}(\xi)=0$.

\begin{definition}
We say that the rigged vector field $\xi$ is orthogonally normal if
\begin{equation}
\label{eqnormal}\widetilde{g}(S(X),S(X))=\widetilde{g}(S^{*\perp}%
(X),S^{*\perp}(X))
\end{equation}
for all $X\in\mathcal{S}^{\zeta}$.
\end{definition}

There are two important cases where the rigged vector field is orthogonally
normal: if $\mathcal{S}^{\zeta}$ is integrable and if $L$ is totally umbilic.
Indeed, if $\mathcal{S}^{\zeta}$ is integrable, then $\xi$ is $\widetilde{g}%
$-irrotational. Therefore $S^{\ast\perp}(X)=S(X)$ for all $X\in\mathcal{S}%
^{\zeta}$ and obviously Equation (\ref{eqnormal}) is satisfied. On the other
hand, if $L$ is totally umbilic, from Remark \ref{remark2}, $S^{\ast\perp
}(X)=2\rho X-S(X)$ for certain $\rho\in C^{\infty}(L)$ and all $X\in
\mathcal{S}^{\zeta}$ and Equation (\ref{eqnormal}) can be easily checked.

Now, we state a formula relating the null sectional curvature and the
$\widetilde{g}$-sectional curvature of planes containing $\xi$ in the case of
being orthogonally normal.

\begin{theorem}
\label{diferenciacurvaturas2} Let $M$ be a Lorentzian manifold, $L$ a null
hypersurface and $\zeta$ a rigging for $L$. Suppose that its rigged vector
field $\xi$ is orthogonally normal. If $\Pi=span(X,\xi)$, where $X\in
\mathcal{S}^{\zeta}$ is a unitary vector, then
\begin{align*}
\mathcal{K}_{\xi}(\Pi)-\widetilde{K}(\Pi)  &  =\tau(\xi)B(X,X)-\widetilde
{g}(\widetilde{\nabla}_{X}\widetilde{\nabla}_{\xi}\xi,X)+\widetilde
{g}(X,\widetilde{\nabla}_{\xi}\xi)^{2}\\
&  +\frac{1}{2}\left(  \widetilde{g}(S^{2}(X),X)-\widetilde{g}%
(S(X),S(X)\right)  .
\end{align*}

\end{theorem}

\begin{proof}
Applying Lemma \ref{diferenciacurvaturas}, we have
\begin{align*}
\widetilde{g}\left(  R_{X\xi}^{L}\xi-\widetilde{R}_{X\xi}\xi,X\right)   &
=\widetilde{g}\left(  (\widetilde{\nabla}_{X}D^{L})(\xi,\xi),X\right)
-\widetilde{g}\left(  (\widetilde{\nabla}_{\xi}D^{L})(X,\xi),X\right) \\
&  +\widetilde{g}\left(  D^{L}(X,D^{L}(\xi,\xi)),X\right)  -\widetilde
{g}\left(  D^{L}(\xi,D^{L}(X,\xi)),X\right)  .
\end{align*}

We compute each term. For the first one,
\begin{align*}
\widetilde{g}\left(  (\widetilde{\nabla}_{X}D^{L})(\xi,\xi),X\right)   &
=\widetilde{g}(\widetilde{\nabla}_{X}D^{L}(\xi,\xi),X)+d\omega(\widetilde
{\nabla}_{X}\xi,X)\\
&  =-\tau(\xi)\widetilde{g}(S(X),X)-\widetilde{g}(\widetilde{\nabla}%
_{X}\widetilde{\nabla}_{\xi}\xi,X)+d\omega(\widetilde{\nabla}_{X}\xi,X).
\end{align*}

The second one is
\[
\widetilde{g}\left(  (\widetilde{\nabla}_{\xi}D^{L})(X,\xi),X\right)
=-\widetilde{g}(D^{L}(X,\xi),\widetilde{\nabla}_{\xi}X)-\widetilde{g}%
(D^{L}(\widetilde{\nabla}_{\xi}X,\xi),X),
\]
but
\begin{align*}
\widetilde{g}(D^{L}(X,\xi),\widetilde{\nabla}_{\xi}X)  &  =g(D^{L}%
(X,\xi),\widetilde{\nabla}_{\xi}X)+\widetilde{g}(D^{L}(X,\xi),\xi
)\widetilde{g}(\widetilde{\nabla}_{\xi}X,\xi)\\
&  =-\frac{1}{2}\left(  \widetilde{g}(\xi,\widetilde{\nabla}_{\xi}X)(L_{\xi
}\widetilde{g})(X,\xi)+d\omega(X,\widetilde{\nabla}_{\xi}X)\right)
+\tau(X)\widetilde{g}(S(\xi),X)\\
&  =\frac{1}{2}\widetilde{g}(S(\xi),X)^{2}-\frac{1}{2}d\omega(X,\widetilde
{\nabla}_{\xi}X)+\tau(X)\widetilde{g}(S(\xi),X).
\end{align*}

Therefore,
\begin{align*}
\widetilde{g}\left(  (\widetilde{\nabla}_{\xi}D^{L})(X,\xi),X\right)   &
=-\frac{1}{2}\widetilde{g}(S(\xi),X)^{2}+\frac{1}{2}d\omega(X,\widetilde
{\nabla}_{\xi}X)-\tau(X)\widetilde{g}(S(\xi),X)\\
&  +\frac{1}{2}\left(  d\omega(\widetilde{\nabla}_{\xi}X,X)+\widetilde
{g}(\widetilde{\nabla}_{\xi}X,\xi)d\omega(\xi,X)\right) \\
&  =-\widetilde{g}(X,S(\xi))^{2}-\tau(X)\widetilde{g}(S(\xi),X).
\end{align*}

The third one is zero due to Corollary \ref{corformula}. The last one is
\begin{align*}
\widetilde{g}\left(  D^{L}(\xi,D^{L}(X,\xi)),X\right)   &  =-\frac{1}%
{2}\left(  d\omega(D^{L}(X,\xi),X)+\widetilde{g}(D^{L}(X,\xi),\xi)d\omega
(\xi,X)\right) \\
&  =-\frac{1}{2}\left(  \widetilde{g}(\widetilde{\nabla}_{D^{L}(X,\xi)}%
\xi,X)-\widetilde{g}(D^{L}(X,\xi),\widetilde{\nabla}_{X}\xi)-\tau
(X)\widetilde{g}(S(\xi),X)\right) \\
&  =-\frac{1}{2}\widetilde{g}(\widetilde{\nabla}_{D^{L}(X,\xi)}\xi,X)-\frac
{1}{4}d\omega(X,\widetilde{\nabla}_{X}\xi)+\frac{1}{2}\tau(X)\widetilde
{g}(S(\xi),X).
\end{align*}

Now, using Formula (\ref{eqcurv}),
\begin{align*}
\widetilde{g}\left(  R_{X\xi}\xi-\widetilde{R}_{X\xi}\xi,X\right)   &
=-\tau(\xi)\widetilde{g}(S(X),X)-\widetilde{g}(\widetilde{\nabla}%
_{X}\widetilde{\nabla}_{\xi}\xi,X)+\widetilde{g}(X,S(\xi))^{2}\\
&  +\frac{1}{2}\tau(X)\widetilde{g}(S(\xi),X)+\frac{1}{2}\widetilde
{g}(\widetilde{\nabla}_{D^{L}(X,\xi)}\xi,X)+\frac{3}{4}d\omega(\widetilde
{\nabla}_{X}\xi,X).
\end{align*}

Finally, we use that $\xi$ is orthogonally normal to compute the last part of
the above formula. Taking into account Formulas (\ref{tensordiferencia}),
(\ref{operadorT}) and Corollary \ref{corformula}
\begin{align*}
&  \frac{1}{2}\widetilde{g}(\widetilde{\nabla}_{D^{L}(X,\xi)}\xi,X)+\frac
{3}{4}d\omega(\widetilde{\nabla}_{X}\xi,X)=\frac{1}{2}\widetilde{g}%
(D^{L}(X,\xi),S^{\ast}(X))+\frac{3}{4}d\omega(S(X),X)\\
&  =-\frac{1}{4}d\omega(X,S^{\ast\perp}(X))+\frac{3}{4}d\omega(S(X),X)-\frac
{1}{2}\tau(X)\widetilde{g}(S(\xi),X)\\
&  =\frac{1}{2}\left(  \widetilde{g}(S^{2}(X),X)-\widetilde{g}%
(S(X),S(X)\right)  -\frac{1}{2}\tau(X)\widetilde{g}(S(\xi),X),
\end{align*}
and we obtain the desired result.
\end{proof}

\begin{corollary}
\label{Riccixi}Let $L$ be a null hypersurface and $\zeta$ a rigging for it.
Suppose that its rigged vector field $\xi$ is orthogonally normal. Then
\[
Ric(\xi)=\widetilde{Ric}(\xi)+\tau(\xi)H-\widetilde{\operatorname{div}%
}\widetilde{\nabla}_{\xi}\xi+\frac{1}{2}\left(  tr(S^{2})-|S^{\perp}%
|^{2}\right)  ,
\]
where $tr$ denotes the trace and $|S^{\perp}|^{2}=\sum_{i=3}^{n}\widetilde
{g}(S(e_{i}),S(e_{i}))$, being $\{e_{3},\ldots,e_{n}\}$ an orthonormal basis
of $\mathcal{S}^{\zeta}$.
\end{corollary}

Observe that the last part of the formula in Theorem
\ref{diferenciacurvaturas2} and Corollary \ref{Riccixi} has sign. Indeed,
using the Cauchy-Schwarz inequality, $\widetilde{g}(S^{2}(X),X)\leq
\widetilde{g}(S(X),S(X)).$

Suppose now that the screen distribution $\mathcal{S}^{\zeta}$ is integrable.
We can consider a leaf of $\mathcal{S}^{\zeta}$ as a submanifold of $(M,g)$ or
$(L,\widetilde{g})$. In the first case, we know that the induced Levi-Civita
connection is $\nabla^{\ast}$ and its second fundamental form is
$\mathbb{I}^{\mathcal{S}^{\zeta}}(X,Y)=C(X,Y)\xi+B(X,Y)N$. In the second case,
the induced connection from $(L,\widetilde{g})$ is also $\nabla^{\ast}$ but
its second fundamental form is $\mathbb{I}(X,Y)=B(X,Y)\xi$. Therefore, if we
call $K^{S}$ and $\widetilde{K}^{S}$ the induced sectional curvatures on a
leaf $S$ of $\mathcal{S}^{\zeta}$ from $(M,g)$ and $(L,\widetilde{g})$
respectively, then
\begin{align*}
K^{S}(\Pi)  &  =\widetilde{K}^{S}(\Pi),\\
K(\Pi)  &  =K^{S}(\Pi)-C(X,X)B(Y,Y)-B(X,X)C(Y,Y)\\
&  +2C(X,Y)B(X,Y),\\
\widetilde{K}(\Pi)  &  =\widetilde{K}^{S}(\Pi)-B(X,X)B(Y,Y)+B(X,Y)^{2},
\end{align*}
for any tangent plane $\Pi=span(X,Y)$ to $\mathcal{S}^{\zeta}$. Moreover, from
Proposition \ref{segundaformafund}, Theorem \ref{diferenciacurvaturas2} and
Corollary \ref{Riccixi} we have the following.

\begin{corollary}
\label{Cor2} Let $L$ be a null hypersurface and $\zeta$ a closed rigging for
it. Then

\begin{enumerate}
\item $\mathcal{K}_{\xi}(\Pi)=\widetilde{K}(\Pi)+\tau(\xi)\frac{B(X,X)}%
{g(X,X)}$, where $\Pi=span(\xi,X)$ and $X\in\mathcal{S}^{\zeta}$.

\item $Ric(\xi)=\widetilde{Ric}(\xi)+\tau(\xi)H$.
\end{enumerate}
\end{corollary}

Moreover, an explicit relation between $\widetilde{R}$ and $R^{L}$ can be
given. For this, recall that if we consider a closed rigging, then $C$ is a
symmetric tensor and using point 3 and 4 of Corollary \ref{corformula},
$C(\xi,X)=-\tau(X)$ for all $X\in\mathcal{S}^{\zeta}$. We need a previous lemma.

\begin{lemma}
\label{lemaD} Let $L$ be a null hypersurface and $\zeta$ a closed rigging for
it. Take $U,V\in\mathfrak{X}(L)$ and $X\in\mathcal{S}^{\zeta}$, then

\begin{enumerate}
\item $\widetilde{\nabla}_{U}B=\nabla^{L}_{U}B$.

\item The tensor $D^{L}=\nabla^{L}-\widetilde{\nabla}$ is given by
\begin{align*}
D^{L}(U,X)  &  =\left(  C(U,X)-B(U,X)\right)  \xi,\\
D^{L}(U,\xi)  &  =-\tau(U)\xi.
\end{align*}

\item The derivative of $D^{L}$ with respect to $\widetilde{\nabla}$ is given
by%
\begin{align*}
\left(  \widetilde{\nabla}_{U}D^{L}\right)  (V,X)  &  =\left(  \left(
\nabla_{U}^{\ast L}C\right)  (V,X)-\left(  \nabla_{U}^{L}B\right)
(V,X)+\tau(V)B(U,X)\right)  \xi\\
&  -A^{\ast}(U)C(V,X)+A^{\ast}(U)B(V,X)+D^{L}(D^{L}(U,V),X),
\end{align*}%
\begin{align*}
\left(  \widetilde{\nabla}_{U}D^{L}\right)  (V,\xi)  &  =\left(
-U(\tau(V))+\tau(\widetilde{\nabla}_{V}U)+C(V,A^{\ast}(U))-B(V,A^{\ast
}(U))\right)  \xi\\
&  +\tau(V)A^{\ast}(U).
\end{align*}

\end{enumerate}
\end{lemma}

\begin{proof}
To prove the first point just take into account that $B(X,\xi)=0$. For the
second point apply Formula (\ref{remarkD}) and Proposition \ref{derivadadeLie}
and \ref{segundaformafund}. Third point is a straightforward computation.
\end{proof}

\begin{theorem}
\label{Teor3}Let $M$ be a Lorentzian manifold, $L$ a null hypersurface and
$\zeta$ a closed rigging for it. Take $U,V\in\mathfrak{X}(L)$ and
$X\in\mathcal{S}^{\zeta}$. Then
\begin{align*}
R_{UV}^{L}X-\widetilde{R}_{UV}X  &  =\left(  g(R_{UV}X,N)-g(R_{UV}%
X,\xi)\right)  \xi\\
&  +C(U,X)A^{\ast}(V)-C(V,X)A^{\ast}(U)\\
&  +B(U,X)\nabla_{V}\xi-B(V,X)\nabla_{U}\xi,\\
R_{UV}^{L}\xi-\widetilde{R}_{UV}\xi &  =g(R_{UV}\xi,N)\xi-\tau(U)A^{\ast
}(V)+\tau(V)A^{\ast}(U).
\end{align*}

\end{theorem}

\begin{proof}
The first formula follows using Lemma \ref{diferenciacurvaturas} and
\ref{lemaD} and Formulas (\ref{eqcurvatura1}) and (\ref{GC2}). We can get the
second one using again Lemma \ref{diferenciacurvaturas} and \ref{lemaD} and
Formulas (\ref{ABadjunto}) and (\ref{GC3}).
\end{proof}

More accurated relations can be obtained if $L$ is totally geodesic.

\begin{corollary}
\label{Teor5} Let $L$ be a totally geodesic null hypersurface and $\zeta$ a
closed rigging for it. Given $U,V,W\in\mathfrak{X}(L)$ and $X,Y\in
\mathcal{S}^{\zeta}$ it holds the following.

\begin{enumerate}
\item $R_{UV}W-\widetilde{R}_{UV}W=g(R_{UV}W,N)\xi$, for all $U,V,W\in
\mathfrak{X}(L)$.

\item If $\Pi=span\{X,U\}$ is a tangent plane to $L$, then
\begin{align*}
K(\Pi)  &  =\left(  1+\frac{g(X,X)\widetilde{g}(U,\xi)^{2}}%
{g(X,X)g(U,U)-g(X,U)^{2}}\right)  \widetilde{K}(\Pi)\text{ if $\Pi$ is
spacelike},\\
\mathcal{K}_{\xi}(\Pi)  &  =\widetilde{K}(\Pi)=0\text{ if $\Pi$ is null}.
\end{align*}

\item The Ricci tensor of $\widetilde{g}$ is given by
\begin{align*}
\widetilde{Ric}(X,Y)  &  =Ric(X,Y)-g(R_{\xi X}Y,N)-g(R_{\xi Y}X,N),\\
\widetilde{Ric}(\xi,U)  &  =Ric(\xi,U)=0.
\end{align*}
\

\item If $\widetilde{s}$ and $s$ denote the scalar curvature of $(L,\widetilde
{g})$ and $(M,g)$ respectively, then%
\[
s-\widetilde{s}=4Ric(\xi,N)-2K\left(  span(\xi,N)\right)  .
\]

\end{enumerate}
\end{corollary}

\begin{proof}
The first point follows inmmediately from the above theorem and Gauss-Codazzi
equations. Since $(L,\widetilde{g})$ is locally a direct product
$\mathbb{R}\times S$ with $\xi$ identified with $\partial_{r}$ (see Theorem
\ref{descomposicion} in the next section), it is obvious that $\widetilde
{Ric}(\xi,U)=0$. The rest is a straightforward computation.
\end{proof}

\section{Applications}

\label{seccionaplicaciones} We show several new results on null hypersurfaces
in order to illustrate our approach. The first one is a new result on the
properties of compact totally umbilic null hypersurfaces which  shows that a mild
curvature condition ensures that it is totally geodesic.  We study conditions
in which more specific Riemannian structures can be introduced on a null
hypersurface, such as twisted, warped or direct product metric, which should be considered as a tool itself.
We also study null conjugate points showing that it is possible to translate the problem to a
Riemannian one. We finish this section showing a new feature of the multiplicity of null conjugate points.

We say that a Lorentzian manifold satisfies the reverse null convergence
condition if $Ric(u)\leq0$ for any null vector $u\in TM$.
Although the opposite inequality is the usual in physical applications,  observe that the reverse null convergence condition
includes the important family of Ricci-flat spacetimes.

\begin{theorem}
\label{Th. 2}Let $M$ be an orientable Lorentzian manifold with dimension $n>2$
which obeys the reverse null convergence condition. If there exists a timelike
conformal vector field on $M$, then any compact totally umbilic null
hypersurface is totally geodesic.
\end{theorem}

\begin{proof}
The case $n=3$ is proven in Corollary \ref{Cor1} below, so we suppose $n\geq
4$. Let $L$ be a null umbilic hypersurface with $B=\rho g$ and $\xi
\in\mathfrak{X}(L)$ the rigged vector field of the timelike conformal vector
field. From Proposition \ref{derivadadeLie}, $\left(  L_{\xi}\widetilde
{g}\right)  (X,Y)=-2\rho\widetilde{g}(X,Y)$ for all $X,Y\in\mathcal{S}$, thus
\begin{align*}
\widetilde{g}(S^{2}(X),X)  &  =-\widetilde{g}(S(X),S(X))-2\rho\widetilde
{g}(S(X),X)\\
&  =-\widetilde{g}(S(X),S(X))+2\rho^{2}\widetilde{g}(X,X),
\end{align*}
where, as above, $S(X)=\widetilde{\nabla}_{X}\xi$, Therefore $tr(S^{2}%
)=-|S^{\perp}|^{2}+2(n-2)\rho^{2}$ and applying Lemma \ref{campoU} and
Corollary \ref{Riccixi},
\[
Ric(\xi)-\widetilde{Ric}(\xi)=-\widetilde{\operatorname{div}}\widetilde
{\nabla}_{\xi}\xi+tr(S^{2})-(n-2)\rho^{2}.
\]

Integrating respect to $\widetilde{g}$,
\[
\int_{L}Ric(\xi)=\int_{L}\widetilde{Ric}(\xi)+tr(S^{2})-(n-2)\rho^{2},
\]
and using the Bochner formula $\int_{L}\widetilde{Ric}(\xi)+tr(S^{2})=\int
_{L}tr(S)^{2}=\int_{L}(n-2)^{2}\rho^{2}$, we get
\[
\int_{L}Ric(\xi)=\int_{L}(n-2)(n-3)\rho^{2}.
\]

Since $M$ holds the reverse null convergence condition, $\rho=0$ so $L$ is
totally geodesic.
\end{proof}

\begin{example}
This example holds the assumptions of the above theorem. Take the torus
\[
\mathbb{T}^{n}=\left(  \mathbb{S}^{1}\times\ldots\times\mathbb{S}^{1}
,dx_{1}dx_{2}+\sum_{i=3}^{n}dx_{i}^{2}\right).
\]

It is flat and
$L=\{x\in\mathbb{T}^{n}:x_{2}=0\}$ is a compact and totally geodesic null hypersurface.
\end{example}

The following theorem says that the local structure of a totally umbilic null
hypersurface, if we consider the induced metric $\widetilde{g}$ from a closed
rigging, is a twisted product.

\begin{theorem}
\label{descomposicion} Let $M$ be a Lorentzian manifold, $L$ a totally umbilic
null hypersurface and $\zeta$ a closed rigging for $L$. Given $p\in L$,
$(L,\widetilde{g})$ is locally isometric to a twisted product $(\mathbb{R}%
\times S,dr^{2}+\lambda^{2}g|_{S})$, where the rigged vector field $\xi$ is
identified with $\partial_{r}$, $S$ is the leaf of $\mathcal{S}^{\zeta}$
through $p$ and
\[
\lambda(r,q)=exp\left(  -\int_{0}^{r}\frac{H(\phi_{s}(q))}{n-2}ds\right),
\]
being $\phi$ the flow of $\xi$. In particular, $dH$ is proportional to
$\omega$ if and only if $(L,\widetilde{g})$ is locally isometric to a warped
product and $L$ is totally geodesic if and only if $(L,\widetilde{g})$ is
locally isometric to a direct product.

Moreover, if $L$ is simply connected and $\xi$ is complete, the above
decomposition is global.
\end{theorem}

\begin{proof}
For simplicity, we suppose that $\xi$ is complete. Since $d\omega=0$, Cartan
formula implies $L_{\xi}\omega=0$, so the flow $\phi$ of $\xi$ is foliated,
that is, $\phi_{r}(S_{q})=S_{\phi_{r}(q)}$ for all $q\in L$ and $r\in
\mathbb{R}$, being $S_{q}$ the leaf of $\mathcal{S}^{\zeta}$ through $q$.
Using this, it is easy to check that $\phi:\mathbb{R}\times S_{p}\rightarrow
L$ is onto and a local diffeomorphism.
Recall that $\{0\}\times S_p$ is identified with $S_p$ itself.
From Proposition \ref{derivadadeLie},
$\left(  L_{\xi}\widetilde{g}\right)  (X,Y)=-2\frac{H}{n-2}\widetilde{g}(X,Y)$
for all $X,Y\in\mathcal{S}^{\zeta}$. Therefore $\phi_{r}:S_{q}\rightarrow
S_{\phi_{r}(q)}$ is a conformal diffeomorphism with conformal factor
$\exp\left(  -2\int_{0}^{r}\frac{H(\phi_{s}(q))}{n-2}ds\right)  $ and it
follows that $\phi^{\ast}(\widetilde{g})=dr^{2}+\lambda^{2}g|_{S_{p}}$, being
$\lambda(r,q)=exp\left(  -\int_{0}^{r}\frac{H(\phi_{s}(q))}{n-2}ds\right)  $.
We show now that $\phi$ is a covering map. Let $\sigma:[0,1]\rightarrow L$ be
a $\widetilde{g}$-geodesic and $(r_{0},x_{0})\in\mathbb{R}\times S_{p}$ a
point such that $\phi(r_{0},x_{0})=\sigma(0)$. We must show that there exists
a lift $\alpha:[0,1]\rightarrow\mathbb{R}\times S_{p}$ of $\sigma$ through
$\phi$ starting at $(r_{0},x_{0})$, \cite[Chapter 7, Theorem 28]{O}. There is a $\widetilde{g}%
$-geodesic $\alpha:[0,s_{0})\rightarrow\mathbb{R}\times S_{p}$, $\alpha
(s)=(r(s),x(s))$, such that $\phi\circ\alpha=\sigma$ and $\alpha
(0)=(r_{0},x_{0})$ because $\phi$ is a local isometry. If we suppose $s_{0}<1
$, there is a geodesic $(r_{1}(s),x_{1}(s))$ such that $\phi(r_{1}%
(s),x_{1}(s))=\sigma(s)$ with $s\in(s_{0}-\varepsilon,s_{0}+\varepsilon)$,
then in the open interval $(s_{0}-\varepsilon,s_{0})$ it holds $\phi
(r(s),x(s))=\phi((r_{1}(s),x_{1}(s))$. Differentiating and using that $\phi$
is foliated, it is easy to see that $r_{1}(s)-r(s)=c\in\mathbb{R} $.
Therefore, it exists $\lim_{s\rightarrow s_{0}}\alpha(s)$ and the
$\widetilde{g}$-geodesic $\alpha$ is extendible.
\end{proof}

This result is applied in Theorem \ref{Teor1} to prove a property concerning the multiplicity of null conjugate point.

\begin{remark}
\label{Nota1}Locally, it always exists a closed timelike vector field, so we
can apply the above theorem to any small enough open set in a totally umbilic
null hypersurface.
\end{remark}

We can also obtain a global decomposition assuming the existence of a timelike
gradient field on $M$ instead of the simply connectedness of $L$. Indeed,
suppose that $f\in C^{\infty}(M)$ is a function with $\zeta=\nabla f$
timelike. If $\gamma:\mathbb{R}\rightarrow L$ is an integral curve of $\xi$,
then $f(\gamma(t))$ is increasing (or decreasing), and since $f$ is constant
along the leaves of $\mathcal{S}^{\zeta}$, $\gamma$ intersects any leaf of
$\mathcal{S}^{\zeta}$ at only one point. Therefore, the map $\phi$ used in the
above proof is injective and $L$ splits globally as $\mathbb{R}\times S$.
Recall that in a stably causal space it always exists a timelike gradient vector field.

\begin{remark}
\label{remarkcompacto} Compactness is an obstruction to get the global
decomposition of a totally umbilic null hypersuperface. Even more, a timelike
gradient field prevents the existence of compact null hypersurfaces (not
necessarily totally umbilic). In fact, in $L$ we can decompose $\nabla
f=X+a\xi+bN$, being $X\in\mathcal{S}^{\zeta}$ and $a,b\in C^{\infty}(L)$. Now,
$\widetilde{\nabla}(f\circ i)=X+b\xi$ and, by compactness, there is a point in
$L$ where $\widetilde{\nabla}(f\circ i)=0$, but then $\nabla f$ is null in
this point, which is a contradiction.
\end{remark}

We give now an example that shows to what extent we can tune the Lorentz metric of the ambient space with the Riemann metric on a null hypersurface, simply by choosing the correct rigging vector field. Suppose we have a null geodesic $\gamma$ through a point $p=\gamma (0)$ with null conjugate points along it. We take the null cone with vertex at $p$ and a suitable rigging vector field on it to show that $\gamma$ is also a geodesic for the rigged metric and both geometries share conjugate points along $\gamma$ as well as its multiplicity. Moreover, we localize the null conjugate point with a different technique used in the literature. In fact, we use Riemannian techniques on the null cone with the rigged Riemannian metric.

Recall that the rigged vector field of a closed rigging is $\widetilde{g}$-geodesic (Proposition \ref{segundaformafund}).

\begin{proposition}
\label{puntosconjugados}Let $(M,g)$ be a Lorentzian manifold, $L$ a null
hypersurface and $\zeta$ a closed rigging for $L$ such that its rigged vector
field $\xi$ is $g$-geodesic. Take $\gamma:I\rightarrow L$ an integral curve of
$\xi$.

\begin{enumerate}
\item If $J$ is a Jacobi field in $(M,g)$ along $\gamma$ with values in $TL$,
then the $\widetilde{g}$-orthogonal projection of $J$ onto $\mathcal{S}^{\zeta}$ is a Jacobi field in
$(L,\widetilde{g})$.

\item If $I=[0,a]$ and $V$ is a Jacobi vector field in $(L,\widetilde{g})$
along $\gamma$ with $V(0)=V(a)=0$, then there exists a Jacobi field $J$ in
$(M,g)$ along $\gamma$ with values in $TL$ such that $J(0)=J(a)=0$.
\end{enumerate}

In particular, $\gamma(a)$ is a conjugate point to $\gamma(0)$ in $(M,g)$ if and only if it is a conjugate point to $\gamma(0)$ in $(L,\widetilde{g})$ and both share the same multiplicity.
\end{proposition}

\begin{proof}
Fix $X\in\mathcal{S}^{\zeta}$ and take an arbitrary $Y\in\mathcal{S}^{\zeta}$.
Using Proposition \ref{Nablas} repeatedly, we get
\begin{align*}
\widetilde{g}(\widetilde{\nabla}_{\xi}\widetilde{\nabla}_{\xi}X,Y)  &
=g(\nabla_{\xi}\widetilde{\nabla}_{\xi}X,Y)=\xi(g(\widetilde{\nabla}_{\xi
}X,Y))-g(\widetilde{\nabla}_{\xi}X,\nabla_{\xi}Y)\\
&  =\xi(g(\nabla_{\xi}X,Y))-g(\nabla_{\xi}X,\nabla_{\xi}Y)=g(\nabla_{\xi
}\nabla_{\xi}X,Y).
\end{align*}

On the other hand, from Equation (\ref{eqcurv}) and Theorem \ref{Teor3}, we
have
\[
\widetilde{g}(\widetilde{R}_{X\xi}\xi,Y)=g(R_{X\xi}\xi,Y),
\]
thus
\begin{equation}
\widetilde{g}(\widetilde{\nabla}_{\xi}\widetilde{\nabla}_{\xi}X+\widetilde
{R}_{X\xi}\xi,Y)=g(\nabla_{\xi}\nabla_{\xi}X+R_{X\xi}\xi,Y). \label{Eq2}%
\end{equation}

Now, suppose that $J$ is a Jacobi field in $(M,g)$ along an integral curve of
$\xi$ with values in $TL$ and call $X$ its projection onto $\mathcal{S}%
^{\zeta}$. Point one easily follows taking into account that $\widetilde
{\nabla}_{\xi}\widetilde{\nabla}_{\xi}X+\widetilde{R}_{X\xi}\xi\in
\mathcal{S}^{\zeta}$, the above formula and $g(\nabla_{\xi}\nabla_{\xi
}J+R_{J\xi}\xi,Y)=g(\nabla_{\xi}\nabla_{\xi}X+R_{X\xi}\xi,Y)$.

Take now $V$ a Jacobi field in $(L,\widetilde{g})$. Since its projection $X$
onto $\mathcal{S}^{\zeta}$ is also a Jacobi field in $(L,\widetilde{g})$, from
equation (\ref{Eq2}) we deduce that
\[
R_{X\xi}\xi+\nabla_{\xi}\nabla_{\xi}X=f\xi,
\]
so $J=X-h\xi$, where $h^{\prime\prime}=f$ and $h(0)=h(a)=0$, is a Jacobi field
in $(M,g)$ with values in $TL$ such that $J(0)=J(a)=0$.

The claim on the multiplicities follows from the above proof.
\end{proof}

Observe that, in general, the projection over $\mathcal{S}^{\zeta}$ of a
Jacobi vector field in $(M,g)$ is not a Jacobi field in $(M,g)$.

Now, we want to use the relationship between conjugate points in $(M,g)$ and
$(L,\widetilde{g})$ to prove a localization result for conjugate points along
a null geodesic. To this end, we need an adapted comparison theorem for
incomplete geodesics in a Riemannian manifold. We introduce the following definitions.

\begin{definition}
Let $(L,\widetilde{g})$ be a Riemannian manifold and $\gamma:(0,a)\rightarrow
L$ an arc length parametrized geodesic.

\begin{itemize}
\item If $\lim_{t\rightarrow0}\widetilde{g}(\widetilde{R}_{X\gamma^{\prime}%
}\gamma^{\prime},Y)$ exists for all parallel vector fields $X,Y$ along
$\gamma$ and orthogonal to $\gamma^{\prime}$, then we say that the tidal force
operator is converging along $\gamma$.

\item If for any Jacobi vector field $J:(0,a)\rightarrow TL$ along $\gamma$
with $\lim_{t\rightarrow0}|J(t)|=0$ it holds $|J(t)|>0$ for all $t\in(0,a)$,
then we say that $\gamma$ has not conjugate points in the interval $(0,a)$.

\item If $\gamma$ has not conjugate points in the interval $(0,a)$ and there
exists a Jacobi vector field $J:(0,a)\rightarrow TL$ with $\lim_{t\rightarrow
0} |J(t)|=\lim_{t\rightarrow a} |J(t)|=0$, we say that $\gamma$ has a
conjugate point in the limit.
\end{itemize}
\end{definition}

If the tidal force operator is converging along $\gamma$, then the Jacobi
equation, written in coordinates respect to a fixed parallel frame
$\{E_{1},\ldots,E_{n}\}$ along $\gamma$, can be considered for all $t\in[0,a)$
and the existence and uniqueness of the solutions, fixed initial condition for
$t=0$, is guaranteed. In particular, a Jacobi vector field with $\lim
_{t\rightarrow0}|J(t)|=\lim_{t\rightarrow0}|J^{\prime}(t)|=0$ is identically
zero. Thus, if $\gamma$ has not conjugate point in the interval $(0,a)$, we
can ensure the existence of $J_{1},\ldots,J_{n-1}:(0,a)\rightarrow TL$ Jacobi
fields such that $\lim_{t\rightarrow0}|J(t)|=0$, $\lim_{t\rightarrow
0}|J^{\prime}(t)|$ exists and $J_{1}(t),\ldots,J_{n-1}(t)$ is a basis of
$\gamma^{\prime\perp}$ for all $t\in(0,a)$. On the other hand, given
$V:(0,a)\rightarrow TL$ a vector field along $\gamma$ we define the index form
as $I_{t_{0}}(V,V)=\int_{0}^{t_{0}}\widetilde{g}(V^{\prime},V^{\prime
})-\widetilde{g}(\widetilde{R}_{V\gamma^{\prime}}\gamma^{\prime},V)$. This
integral can be diverging or even not exist, but if $J:(0,a)\rightarrow TL$ is
a Jacobi vector field with $\lim_{t\rightarrow0}|J(t)|=0$ and such that
$\lim_{t\rightarrow0}|J^{\prime}(t)|$ exists, then $I_{t_{0}}(J,J)=\widetilde
{g}(J(t_{0}),J^{\prime}(t_{0}))$. Taking all this into account, the proof of
the classical index lemma and the Rauch comparison theorem, see for example \cite{1975CheegerEbin}, can be followed step by step to state the following.

\begin{lemma}
Let $(L,\widetilde{g})$ be a Riemann manifold, $\gamma:(0,a)\rightarrow L$ an
arc length parametrized geodesic without conjugate points in the interval
$(0,a)$ and $J,V:(0,a)\rightarrow TL$ vector fields along $\gamma$ with
$\lim_{t\rightarrow0}|J(t)|=\lim_{t\rightarrow0}|V(t)|=0$, $\lim
_{t\rightarrow0}|J^{\prime}(t)|$ exists, $g(J,\gamma^{\prime})=g(V,\gamma
^{\prime})=0$ and $J(t_{0})=V(t_{0})$ for some $t_{0}\in(0,a)$. If $J$ is a Jacobi vector field, the tidal
force operator is converging along $\gamma$ and $I_{t_{0}}(V,V)$ exists, then
$I_{t_{0}}(J,J)\leq I_{t_{0}}(V,V)$.
\end{lemma}

\begin{theorem}
\label{Rauchadaptado} Let $(L,\widetilde{g})$ and $(\overline{L},\overline
{g})$ be two Riemannian manifolds and $\gamma:(0,a)\rightarrow L$,
$\overline{\gamma}:[0,a]\rightarrow\overline{L}$ two arc length parametrized
geodesics. Suppose that the tidal force operator is converging along $\gamma$
and take $J:(0,a)\rightarrow TL$ and $\overline{J}:[0,a]\rightarrow
T\overline{L}$ two Jacobi vector fields such that $\lim_{t\rightarrow
0}|J(t)|=|\overline{J}(0)|=0$, $\lim_{t\rightarrow0}|J^{\prime}(t)|=|\overline
{J}^{\prime}(0)|$ and $\widetilde{g}(J(t),\gamma^{\prime}(t))=\overline
{g}(\overline{J}(t),\overline{\gamma}^{\prime}(t))=0$ for all $t\in(0,a)$.
Then , the following statements hold.

\begin{itemize}
\item If $\overline{\gamma}$ has not conjugate points to $\overline{\gamma
}(0)$ and $K(span(\gamma^{\prime},v))\leq\overline{K}(span(\overline{\gamma
}^{\prime},\overline{v}))$ for all $v\in\gamma^{\prime\perp}$ and all
$\overline{v}\in\overline{\gamma}^{\prime}(t)$, then $|\overline{J}%
(t)|\leq|J(t)|$ for all $t\in(0,a)$.

\item If $\gamma$ has not conjugate points in the interval $(0,a)$ and
$K(span(\gamma^{\prime},v))\geq\overline{K}(span(\overline{\gamma}^{\prime
},\overline{v}))$, then $|\overline{J}(t)|\geq|J(t)|$.
\end{itemize}
\end{theorem}

Now, we are able to give a new proof of the following localization result for conjugate
point along a null geodesic using Riemannian techniques. Point 1 of this theorem was proven using a different technique in \cite{Harris}.

\begin{theorem}
\label{Teor4}Let $M$ be a Lorentzian manifold and $\gamma:[0,a]\rightarrow M$
a null geodesic such that $\gamma(a)$ is the first conjugate point to
$\gamma(0)$ along $\gamma$. Let $c>0$ be a constant.

\begin{enumerate}
\item If $c^{2}\leq\mathcal{K}_{\gamma^{\prime}}(\Pi)$ for all null plane
containing $\gamma^{\prime}$, then $a\leq\frac{\pi}{c}$.

\item If $\mathcal{K}_{\gamma^{\prime}}(\Pi)\leq c^{2}$ for all null plane
containing $\gamma^{\prime}$, then $\frac{\pi}{c}\leq a$.
\end{enumerate}
\end{theorem}

\begin{proof}
Call $p=\gamma(0)$ and suppose $\gamma(t)=\exp_{p}(tu)$ for all $t\in[0,a]$,
where $u\in T_{p}M$ is a null vector. Without loss of generality, we can
suppose that $\gamma$ does not intersect itself. In this case, there exist
open subsets $\widehat{\theta}\subset T_{p}M$ and $\theta\subset M$ with
$tu\in\widehat{\theta}$ and $\gamma(t)\in\theta$ for all $t\in[0,a)$  such
that $exp_{p}:\widehat{\theta}\rightarrow\theta$ is a diffeomorphism. Now,
$L=\exp_{p}(\widehat{\theta}\cap\widehat{C}_{p})$ is a null hypersurface which contains $\gamma(t)$ for all $t\in(0,a)$.

Take $e\in T_{p}M$ a timelike unitary vector with $g_{p}(u,e)=1$ and define $\widehat{f}:\widehat{\theta} \rightarrow \mathbb{R}$ the function
given by $\widehat{f}(v)=g_{p}(v,e)$. Then, $\zeta=\nabla f$, where $f=\widehat{f} \circ exp_{p}^{-1}:\theta \rightarrow \mathbb{R}$, is a
rigging for $L$ and its rigged vector field is $\xi=\frac{P}{f}$ where $P$ is the position vector field at $p$ restricted to $L$. Moreover,
it is straightforward to see that $\xi$ is $g$- geodesic and that $\gamma:(0,a)\rightarrow L$ is an integral curve of $\xi$.

Since $\gamma(a)$ is the first conjugate point to $\gamma(0)$, there is
$J:[0,a]\rightarrow TL$ a Jacobi field along $\gamma$ with $J(0)=0$, $J(a)=0$,
$J(t)$ non-parallel to $\gamma^{\prime}(t)$ and $J(t)\neq0$ for all
$t\in(0,a)$. Applying Proposition \ref{puntosconjugados}, its projection $X$
onto $\mathcal{S}^{\zeta}$ is a Jacobi field in the Riemannian manifold
$(L,\widetilde{g})$ over the $\widetilde{g}$-geodesic $\gamma:(0,a)\rightarrow
L$ (recall that $\xi$ is $\widetilde{g}$-geodesic because the rigging is closed). Taking into account that $\widetilde{g}(X,X)=g(J,J)$ and that
$\widetilde{g}(\widetilde{\nabla}_{\xi}X,\widetilde{\nabla}_{\xi}%
X)=g(\nabla_{\xi}J,\nabla_{\xi}J)$ (Proposition \ref{Nablas}), it follows that
$\lim_{t\rightarrow0}\left\vert X(t)\right\vert =\lim_{t\rightarrow
a}\left\vert X(t)\right\vert =0$, $X(t)\neq0$ for all $t\in(0,a)$ and
$\lim_{t\rightarrow0}|\widetilde{\nabla}_{\xi}X|$ exists.

Now, if $V:[0,a]\rightarrow TL $ is a $g$-parallel vector field along $\gamma
$, using Proposition \ref{Nablas}, it is easy to show that its projection onto
$\mathcal{S}^{\zeta}$ is $\widetilde{g}$-parallel. Applying this, Equation
(\ref{eqcurv}) and Theorem \ref{Teor3} we can check that the tidal force
operator is converging along $\gamma$.

Finally, by Corollary \ref{Cor2}, we have $c^{2}\leq\widetilde{K}(\Pi)$ ( or
$\widetilde{K}(\Pi)\leq c^{2}$) for any plane containing $\xi$ and a standard
application of Theorem \ref{Rauchadaptado} gives us the result.
\end{proof}

Observe that in $(M,g)$ the maximum multiplicity of a null conjugate point $\gamma (a)$ of $\gamma (0)$ along a null geodesic $\gamma$ is $n-2$. This is because if $J$ is a Jacobi field along $\gamma$  with $J(0)=J(a)=0$, it must be orthogonal to $\gamma$ but not proportional to it, \cite[p. 218]{O}. In $(L,\widetilde{g})$ the maximum multiplicity of $\gamma (a)$ is also $n-2$ because $(L,\widetilde{g})$ is a $(n-1)$-dimensional Riemannian manifold. In \cite{2000Flores} it was proved that null conjugate points in Robertson-Walker spaces have maximum multiplicity. After, we proved in \cite{GutOlea15} that it is also true for a null geodesic in a generalized Robertson-Walker space provided it is contained in a totally umbilic null cone, which suggests that it could be a general feature of totally umbilic null cones itself. The following result shows that this is the case.

\begin{theorem} \label{Teor1}
Let $M$ be a Lorentzian manifold and $\gamma:[0,a]\rightarrow M$
a null geodesic such that $\gamma(a)$ is the first conjugate point to
$\gamma(0)$ along $\gamma$. If the null cone with vertex at $\gamma(0)$ containing $\gamma$ is totally umbilic, then $\gamma(a)$ has maximum multiplicity.
\end{theorem}

\begin{proof}
We consider the same construction as in the above proof. Take a $\widetilde{g}$-Jacobi field $J$ along $\gamma$ such that $J \in \mathcal{S}^{\zeta}$. Using  Theorem \ref{descomposicion}
and the formulas for the curvature in a twisted product (see \cite{1993PongeReckziegel}), it is easy to
see that $\widetilde{R}_{J\xi}\xi=\frac{\widetilde{Ric}(\xi) }{n-2} J$. Thus, the multiplicity of $\gamma(a)$ as a $\widetilde{g}$-conjugate
point is maximal, that is, $n-2$ and by Proposition \ref{puntosconjugados} the same is true in $(M,g)$.
\end{proof}

\section{Null surfaces}

\label{NullSurfaces}In this section we point out how can be adapted section
\ref{seccionproyeccion} to null submanifold of arbitrary dimension. After
that, we particularize to null surfaces.

Let $\Sigma$ be a $k$ dimensional null submanifold of a time-orientable
Lorentzian manifold $(M,g)$. The main difference from hypersurfaces is that
now, in general, $T\Sigma^{\perp}$ is not contained in $T\Sigma$ and has
dimension greater than one, so we take $Rad(T\Sigma)=T\Sigma\cap
T\Sigma^{\perp}$, which is a one dimensional null distribution in $\Sigma$.

\begin{definition}
A rigging for $\Sigma$ is a vector field $\zeta$ defined in some neighborhood
of $\Sigma$ such that $\zeta\notin Rad(T\Sigma)^{\perp}$.
\end{definition}

We take the screen distribution and the transversal screen distribution given
by
\begin{align*}
\mathcal{S}^{\zeta}  &  =T\Sigma\cap\zeta^{\perp},\\
tr\mathcal{S}^{\zeta}  &  =T\Sigma^{\perp}\cap\zeta^{\perp}.
\end{align*}

We construct $\widetilde{g}$ and the rigged vector field $\xi$ as in section
\ref{seccionproyeccion}. It holds $TM=T\Sigma\oplus tr\mathcal{S}^{\zeta
}\oplus span\{N\}$, where $N=\zeta-\frac{1}{2}g(\zeta,\zeta)\xi$. Thus, given
$U,V\in\mathfrak{X}(\Sigma)$, we can decompose
\[
\nabla_{U}V=\nabla_{U}^{\Sigma}V+h^{s}(U,V)+B(U,V)N,
\]
where $\nabla_{U}^{\Sigma}V\in T\Sigma$ and $h^{s}(U,V)\in tr\mathcal{S}%
^{\zeta}$. $h^{s}$ is called the \textit{screen second fundamental form} and
$B$ the \textit{null second fundamental form}. If $h^{s}=B=0$, then $\Sigma$
is totally geodesic, which is an intrisic property. The nullity of $B$ is an
intrisic property of the null submanifold too, but it is weaker than being
totally geodesic.

Since $\nabla^{\Sigma}_{U}\xi\in T\Sigma$, we decompose $\nabla^{\Sigma}%
_{U}\xi=-A^{*}(U)-\tau(U)\xi$, where $A^{*}(U)\in\mathcal{S}^{\zeta}$ and
$\tau$ is a one form. It holds $A^{*}(\xi)=0$ and $B(\xi,\cdot)=0$, thus
$\nabla_{\xi}\xi=h^{s}(\xi,\xi)-\tau(\xi)\xi$.

Unlike hypersuperfaces, the induced null curvature $\mathcal{K}_{\xi}^{\Sigma
}$ defined as $\mathcal{K}_{\xi}^{\Sigma}(\Pi)=g(R_{X\xi}^{\Sigma}\xi,X)$
where $\Pi=span(X,\xi)$ with $X\in\mathcal{S}^{\zeta}$ unitary, does not
coincide with the ambient null sectional curvature $\mathcal{K}_{\xi}$.
Indeed, it holds
\begin{align}
\mathcal{K}_{\xi}(\Pi)  &  =\mathcal{K}_{\xi}^{\Sigma}(\Pi)+g\left(
h^{s}(X,\xi),h^{s}(X,\xi)\right)  -g\left(  h^{s}(X,X),h^{s}(\xi,\xi)\right)
,\label{curvaturaluzsuperficie}\\
\mathcal{K}_{\xi}^{\Sigma}(\Pi)  &  =\tau(\xi)B(X,X)+g\left(  \left(
\nabla_{\xi}^{\Sigma}A^{\ast}\right)  (X),X\right)  -g\left(  \left(
\nabla_{X}^{\Sigma}A^{\ast}\right)  (\xi),X\right)  ,
\label{curvaturaluzsuperficie2}%
\end{align}
where $\Pi=span(X,\xi)$ and $X\in\mathcal{S}^{\zeta}$ is unitary.

Proposition \ref{diferenciaconexiones} and all derived results still hold in
the case of a null submanifold of arbitrary dimension. In particular,
$\tau(U)=g(\xi,\nabla_{U}\zeta)$ for all $U\in\mathfrak{X}(\Sigma)$. Thus, if
the rigging is conformal, then $\tau(\xi)=0$, but unlike hypersurfaces, $\xi$
does not need to be geodesic.

In \cite{Birman} it is shown that the integral of the curvature of an
orientable, time-orientable and compact Lorentzian surface is zero. Theorem
\ref{teoremasuperficiesluz} states an analogous result for null surfaces of a
Lorentzian manifold. Recall that for a null surface $\Sigma$, the null second
fundamental form $B$ can be considered as a function on $\Sigma$, as well as
the one-form $\tau$, the induced null curvature $\mathcal{K}_{\xi}%
^{\Sigma}$ and the Gauss curvature of $\left(  \Sigma,\widetilde{g}\right)  $.

\begin{theorem}
\label{teoremasuperficiesluz} Let $M$ be a Lorentzian manifold and $\Sigma$ a
null surface. If $\zeta$ is a rigging for $\Sigma$, then it holds
\[
\mathcal{K}_{\xi}^{\Sigma}-\widetilde{K}=\tau(\xi)B-\widetilde{div}%
\widetilde{\nabla}_{\xi}\xi.
\]

In particular, if $\Sigma$ is compact and orientable and the rigging $\zeta$
is conformal, then
\[
\int_{\Sigma}\mathcal{K}_{\xi}^{\Sigma}d\widetilde{g}=0.
\]

\end{theorem}

\begin{proof}
Suppose that $\Pi=span(X,\xi)$, with $X$ unitary, is a tangent plane to
$\Sigma$. Since $dim\,\Sigma=2$, it follows that $\xi$ is orthogonally normal
and $\widetilde{\nabla}_{X}\xi=aX$, $\widetilde{\nabla}_{\xi}\xi=bX$, so%
\[
\widetilde{g}(S^{2}(X),X)-\widetilde{g}(S(X),S(X))=0,
\]
being $S(X)=\widetilde{\nabla}_{X}\xi$. Moreover, $\widetilde{div}%
\widetilde{\nabla}_{\xi}\xi=g(\widetilde{\nabla}_{X}\widetilde{\nabla}_{\xi
}\xi,X)-g(\widetilde{\nabla}_{\xi}\xi,X)^{2}$. The proof of Theorem
\ref{diferenciacurvaturas2} is still valid in this situation and therefore we
have
\[
\mathcal{K}_{\xi}^{\Sigma}(\Pi)-\widetilde{K}(\Pi)=\tau(\xi)B(X,X)-\widetilde
{div}\widetilde{\nabla}_{\xi}\xi.
\]

For the second part, observe that the Euler-Poincar\'{e} characteristic of
$\Sigma$ is zero.
\end{proof}

The following corollary extends Theorem \ref{Th. 2} to the three dimensional case.

\begin{corollary}
\label{Cor1}Let $M$ be a three dimensional Lorentzian manifold furnished with
a timelike conformal vector field. If it holds the (reverse) null convergence
condition, then any compact and orientable null surface is totally geodesic.

Moreover, if $Ric(u)\neq0$ for all null vector $u$, then it can not exists any
compact orientable null surface.
\end{corollary}

\begin{proof}
Let $\Sigma$ be a null surface and consider the timelike conformal vector
field $\zeta$ as the rigging field for $\Sigma$ with rigged field $\xi$.
Recall that in this case $h^{s}\equiv0$ and $Ric(\xi)=\mathcal{K}_{\xi
}=\mathcal{K}_{\xi}^{\Sigma}$. If we consider $B$ and $\mathcal{K}_{\xi}$ as
functions on $\Sigma$, Formula (\ref{curvaturaluzsuperficie}) can be written
as $\mathcal{K}_{\xi}=\xi(B)-B^{2}$. By hypothesis, $\mathcal{K}_{\xi}$ has
sign and applying the above theorem $\mathcal{K}_{\xi}=0$. Since $\Sigma$ is
compact, $\xi$ is a complete vector field, thus $B\equiv0$.
\end{proof}

\begin{example}
We already know that Lorentzian Berger spheres $\left(  \mathbb{S}%
^{2n+1},g_{L}\right)  $ do not admit closed embedded null hypersurfaces (example
\ref{ejemploBerger}). In particular, $\left(  \mathbb{S}^{3},g_{L}\right)  $
does not admit compact null surfaces.

However, there exist compact null surfaces in $\left(  \mathbb{S}^{2n+1}%
,g_{L}\right)  $ with $n\geq2$. In fact, consider the Euclidean sphere
$\left(  \mathbb{S}^{2n+1},g_{R}\right)  $ in $\mathbb{C}^{n+1}$ and call $E$
the Hopf vector field given by $E_{(z_{0},\ldots,z_{n})}=(iz_{0},\ldots
,iz_{n})$. Choose\ $r_{0},r_{1}\in\mathbb{R}^{+}$, and $\lambda_{k}%
\in\mathbb{C}$, $k=2,\ldots,n$, with $r_{0}^{2}+r_{1}^{2}=\frac{1}{2}$ and
$|\lambda_{2}|^{2}+\ldots+|\lambda_{n}|^{2}=\frac{1}{2}$. We can parametrize
the surface
\[
\Sigma=\{(z_{0},\ldots,z_{n})\in\mathbb{S}^{2n+1}:|z_{0}|=r_{0}\text{, }%
|z_{1}|=r_{1},\ z_{2}=\lambda_{2},\ldots,\ z_{n}=\lambda_{n}\}
\]
as
\begin{align*}
\Phi:\left(  0,2\pi\right)  \times\left(  0,2\pi\right)   &  \longrightarrow
\mathbb{S}^{2n+1}\\
(t,s)  &  \mapsto\left(  r_{0}e^{it},r_{1}e^{is},\lambda_{2},\ldots
,\lambda_{n}\right)  .
\end{align*}

It holds
\begin{align*}
g_{R}(\Phi_{t},\Phi_{t})  &  =r_{0}^{2},\ \ \ g_{R}(\Phi_{s},\Phi_{s}%
)=r_{1}^{2},\ \ \ g_{R}(\Phi_{t},\Phi_{s})=0,\\
g_{R}(\Phi_{t},E)  &  =r_{0}^{2},\ \ \ g_{R}(\Phi_{s},E)=r_{1}^{2}.
\end{align*}

Therefore, if we take the Lorentzian Berger metric $g_{L}=g_{R}-2\Omega
\otimes\Omega$, being $\Omega$ the $g_{R}$-metrically equivalent one-form to
$E$, then $\Sigma$ is a compact null surface in $\left(  \mathbb{S}%
^{2n+1},g_{L}\right)  $.

Observe that for these surfaces $h^{s}\not \equiv 0$. On the contrary case,
$\int_{\Sigma}\mathcal{K}_{\xi}(\Pi)d\widetilde{g}=0$, being $\Pi$ a tangent
plane to $\Sigma$ and $\xi$ the rigged of $E$. But $\left(  \mathbb{S}%
^{2n+1},g_{L}\right)  $ has strictly positive null sectional curvature.
\end{example}

\begin{corollary}
Let $M$ be a Lorentzian manifold such that $\mathcal{K}(\Pi)<0$ for all null
plane $\Pi$ and $\Sigma$ an orientable null surface. If there exists a
geodesic null vector field $\xi\in\mathfrak{X}(\Sigma)$, then $\Sigma$ can not
be compact.
\end{corollary}

\def\bstname{mn}

\end{document}